\documentclass{elsarticle}

\makeatletter
\def\ps@pprintTitle{%
 \let\@oddhead\@empty
 \let\@evenhead\@empty
 \def\@oddfoot{}%
 \let\@evenfoot\@oddfoot}
\makeatother

\usepackage{color} 
\usepackage{amssymb}
\usepackage{amsthm}
\usepackage{MnSymbol}
\usepackage{pgf,tikz}
\usepackage{caption}
\usepackage[french]{babel}

\usepackage{chngcntr}
\usepackage{lipsum}
\usetikzlibrary{arrows}

\usepackage{enumitem}

\newtheorem{theoreme}{Th\'eor\`eme}[section]
\newtheorem{e-proposition}[theoreme]{Proposition}

\newtheorem{e-definition}[theoreme]{Definition\rm}

\newtheorem{lemme}[theoreme]{Lemme}
\newtheorem{proposition}[theoreme]{Proposition}
\newtheorem{corollaire}[theoreme]{Corollaire}

 \newtheorem{Fait}{\textbf{Fait}}[section]
 
\newtheorem*{property}{\textbf{Propriété}}

\usepackage{calrsfs} 
\usepackage{ulem} 
\usepackage{cancel} 
\usepackage{pgf,tikz}
\usepackage{mathrsfs}
\usetikzlibrary{arrows}
\usepackage{latexsym}
\usepackage{graphicx}
\usepackage{caption} 
\usetikzlibrary{shapes}

\begin{document}
\nocite{*}
\sloppy

\begin{frontmatter}

\title{Sommets fortement critiques d'un tournoi indécomposable}
\author{Rachid Sahbani\fnref{}}
\address{Sfax University, Faculty of Sciences of Sfax, Sfax, Tunisia}
\ead{rachidsahbani5@gmail.com}

\begin{abstract}
Let $T=(V,A)$ be a tournament. For $X\subseteq V$, the subtournament of $T$ induced by $X$ is denoted by $T[X]$. A subset $I$ of $V$ is an interval of $T$ provided that for every $a,b\in I$ and $x\in V\setminus I$, $(a,x)\in A$ if and only if $(b,x)\in A$. For example, $\varnothing $, ${x}$ ($x \in V$) and $V$ are intervals of $T$, called trivial intervals. The tournament $T$ is indecomposable if all its intervals are trivial, otherwise, it is decomposable. 
A critical tournament is an indecomposable tournament $T$ of cardinality $\geqslant 5$ such that every vertex $x$ of $T$ is critical, i.e., the subtournament $T[V(T)\setminus\{x\}]$ is decomposable. Given an  indecomposable tournament $T$, a vertex $x$ of $T$ is  strongly critical, if for every $X\subseteq V(T)$ such that $x\in X$, $\vert X\vert \geqslant 5$ and $T[X]$ is indecomposable, $x$ is a critical vertex of $T[X]$. Let $T$ be an indecomposable tournament and let  $\mathscr{C}(T)$ be the set of the strongly critical vertices of $T$. We prove that, if $T$ is non-critical, then $f(T):=\vert  \mathscr{C}(T)\vert \leqslant 4$, and that the  correspondence $f(T)$ is decreasing from the class of indecomposable and non-critical  tournaments (defined by means of embedding) to $\{0,1,2,3,4\}$. By giving examples, we also verify that the bounds 0 and 4 are optimal. This article is an extract from my master's thesis \cite{mon mastère}.
 
\end{abstract}
\end{frontmatter}
\section*{Résumé} 
 Soit $T=(V,A)$ un tournoi. Pour $X\subseteq V$, le sous-tournoi de $T$ induit par $X$ est noté $T[X]$. Une partie $I$ de $V$ est un intervalle de $T$ lorsque pour tous $a,b\in I$ et $x\in V\setminus I$, $(a,x)\in A$ si et seulement si $(b,x)\in A$. Par exemple, $\varnothing$, $V$ et $\{x\}$ ($x\in V$) sont des intervalles de $T$, appelés intervalles triviaux.  Un tournoi est indécomposable lorsque tous ses intervalles sont triviaux, sinon il est décomposable. Un tournoi critique $T$ est un tournoi indécomposable d'ordre $\geqslant 5$ tel que tout sommet $x$ de $T$ est critique, i.e., le sous-tournoi $T[V(T)\setminus\{x\}]$ est décomposable. \'Etant donné un tournoi indécomposable $T$, un sommet $x$ de $T$ est fortement critique lorsque  pour toute partie $X$ de $V(T)$ telle que $x\in X$, $\vert X\vert \geqslant 5$ et $T[X]$ est indécomposable, $x$ est un sommet critique de  $T[X]$. Soit $T$ un tournoi indécomposable et soit $\mathscr{C}(T)$ l'ensemble des sommets fortement critiques de $T$. Nous démontrons que si $T$ est non critique, alors $f(T):=\vert  \mathscr{C}(T)\vert \leqslant 4$, et que la correspondance $f(T)$ est décroissante de la classe des tournois indécomposables et non critiques (munie du préordre de l'abritement) sur $\{0,1,2,3,4\}$. En construisant des exemples de tailles arbitraires, nous vérifions aussi que les bornes 0 et 4 sont optimales. Cet article est un extrait de mon mémoire de mastère \cite{mon mastère}.

 \section{Introduction}
Un \textit{tournoi} $T$  est un couple $(V,A)$, où $V$ est un ensemble fini, appelé ensemble des \textit{sommets} de $T$, et $A$ est un ensemble de couples de sommets distincts de $T$, appelé ensemble des \textit{arcs } de $T$, vérifiant : pour tous $x\neq y \in V$, $(x, y)\in A$ si et seulement si $(y, x)\notin A$. À chaque partie $X$ de $V$ est associé le sous-tournoi 
$T[X] =(X, A \cap (X \times X))$ \textit{induit } par $X$. Le sous-tournoi induit $T[V\setminus X]$ est aussi noté $T-X$, et est noté $T-x$ lorsque $X=\{x\}$. L'ordre d'un tournoi $T$ est le cardinal de son ensemble de sommets $V(T)$. Soit $T=(V,A)$ un tournoi. Nous introduisons les notations et les notions  suivantes. Pour tous $x\neq y\in V$, $x\longrightarrow y$ signifie que $(x,y)\in A$. Pour  $Y\subset V$ et $x\in V\setminus Y$, $x\longrightarrow Y$ (resp. $Y\longrightarrow x$) signifie que pour tout $y\in Y$, $x\longrightarrow y$ (resp. $y\longrightarrow x$). Pour tout $x\in V$, on note $N^+_T(x)=\{ y\in V: (x,y)\in A\}$ et $N^-_T(x)=\{ y\in V: (y,x)\in A\}$. Un \textit{tournoi transitif} ou un \textit{ordre total} est un tournoi $T$ tel que pour tous $x,y,z\in V(T)$, si $(x,y)\in A(T)$ et $(y,z)\in A(T)$, alors $(x,z)\in A(T)$. Pour tout entier $n\geqslant 2$, l'ordre total usuel $O_n$ est le tournoi défini sur $\{0,\ldots, n-1\}$ par $A(O_n)=\{(i,j): 0\leqslant i< j\leqslant n-1\}$.

 Deux tournois $T=(V,A)$ et $T'=(V',A')$ sont  \textit{isomorphes}, et on écrit $T\simeq T'$, lorsqu'il existe un isomorphisme de $T$ su $T'$, c'est à dire une bijection de $f$ de $V$ su $V'$ telle que pour tous $x,y\in V$, $(x,y)\in A$ si et seulement si $(f(x),f(y))\in A'$. Un tournoi $T$ \textit{abrite} un tournoi $T'$ si $T'$ est isomorphe à un sous tournoi de $T$, sinon $T$ \textit{omet} $T'$. \`A tout tournoi $T=(V, A)$ est associé son tournoi \textit{dual} $T^{\star}=(V, A^{\star})$, où $A^{\star}=\{(y,x): (x,y)\in A\}$. Un tournoi est \textit{autodual} s'il est isomorphe à son dual. 

\'Etant donné un tournoi $T=(V,A)$, une partie $I$ de $V$ est un \textit{intervalle} \cite{F1, F2, I, S.T} (ou {\it clan} 
\cite{er} ou ensemble {\it homogène} \cite{strong intervals,strong interval}  ou {\it module}  \cite{Muller}) de $T$ lorsque pour tout $x\in V\setminus I$,  $x\longrightarrow I$ ou bien $I\longrightarrow x$. Par exemple $\varnothing$, $\{x\}$ où $x\in V$, et $V$ sont des intervalles de $T$, appelés les intervalles triviaux de $T$. Un tournoi, à au moins trois sommets, est {\it indécomposable} \cite{ S.T,I} \index{indécomposable} (ou {\it premier} \cite{CH} ou {\it primitif} \cite{EHR, er}) lorsque tous ses intervalles sont triviaux, sinon il est {\it décomposable}. Par exemple, le 3-cycle $C_3=(\{0,1,2\},\{(0,1),(1,2),(2,1)\})$ est indécomposable alors qu'un ordre total est décomposable.
\'Etant donnés deux tournois isomorphes $T$ et $T'$, si $f$ est un isomorphisme de $T$ sur $T'$, alors une partie $I$ de $V(T)$ est un intervalle de $T$ si et seulement si $f(I)$ est un intervalle de de $T'$. En particulier $T$ est indécomposable si et seulement si $T'$ est indécomposable. Notons aussi qu'un tournoi $T$ et son dual $T^{\star}$ ont les même intervalles, et donc le tournoi $T$ est indécomposable si et seulement si $T^{\star}$ est indécomposable.

Rappelons que, à isomorphisme près, les quatre tournois à 4 sommets sont les tournois $O_4$, $T_4$ et les \textit{diamants} $D_4$ et son dual $D^{\star}_4$,  définis sur $\{0,1,2,3\}$ par:

 $A(O_4)=\{(i,j): 0\leqslant i<j\leqslant3\}$,
 
  $A(T_4)=\{(0,1), (0,2), (1,2), (2,3), (3,0), (3,1) \}$, 
  
   $A(D_4)=\{ (0,1),(1,2), (2,0),(3,0),(3,1),(3,2)\} $.  
   
Clairement, tous les  tournois à 4 sommets sont tous décomposables et seul $C_3$, à isomorphisme près, le tournoi indécomposable à trois sommets.

Maintenant, considérons un tournoi $T=(V,A)$,  à au moins $3$ sommets, avec une partie  $X$ de $V$ telle que $ \vert X\vert\geqslant 3 $ et $T[X]$ est indécomposable. On introduit les parties suivantes de $ V\setminus X $.
\begin{itemize}
\item[$\bullet$] $ Ext(X)$ est l'ensemble des $ x\in V\setminus X$ tels que $T[ X\cup\lbrace x\rbrace ]$ est indécomposable.
\item[$\bullet$] $ \langle X \rangle $ est l'ensemble des $ x\in V \setminus X $ tels que $X$ est un intervalle de $T[X\cup \lbrace x\rbrace ]$.
\item[$\bullet$] Pour tout $ u\in X $, $ X(u)$ est l'ensemble des $ x\in V\setminus X $ tels que   $ \lbrace u,x\rbrace $ est un intervalle de $T[ X\cup \lbrace x\rbrace ]$.
 \end{itemize}
 La famille  $\lbrace Ext(X),\langle X\rangle \rbrace \cup \lbrace X(u),u\in X \rbrace $ est notée  $p_{X}$ et est appelée {\it partition extérieure}  de $T$ induite par $X$.  
  \begin{lemme}\cite{er} \label{ER} Soit $ T=(V,A)$ un tournoi et soit $X$ une partie de $V$ telle que $ \vert X \vert\geqslant 3  $ et $ T[X]$ est  indécomposable. Les éléments non vides de la famille $p_{X}$ forment une partition de $ V\setminus X $. 
 \end{lemme}

Le théorème suivant est un résultat important sur  l'aspect héréditaire ascendant de l'indécomposabilité. 
\begin{theoreme} \cite{S.T} \label{ER +2, pour tournoi}
\'Etant donné un tournoi indécomposable $T$, pour toute  partie $X$ de $V(T)$ telle que $3 \leqslant |X| \leqslant |T| -2$ et $T[X]$ est indécomposable, il existe $x \neq y \in V(T) \setminus X$ tels que $T[X \cup \{x,y\}]$ est indécomposable.
\end{theoreme}

Comme tout tournoi indécomposable, à au moins trois sommets, abrite un tournoi indécomposable à trois  sommets \cite{er}, 
nous obtenons, En appliquant plusieurs fois le théorème~\ref{ER +2, pour tournoi}, ce que suit.  
\begin{corollaire} \label{-1-2, pour tournoi}
Tout tournoi indécomposable d'ordre $n \geqslant 5$ abrite un tournoi indécomposable d'ordre $n-1$ ou $n-2$.
\end{corollaire}
 
 Le corollaire~\ref{-1-2, pour tournoi}, amène alors à introduire la notion de criticité. Un sommet \textit{critique} d'un tournoi indécomposable $T$ est un sommet  de $T$ tel que le tournoi $T-x$ est décomposable. Un tournoi indécomposable, à au moins cinq sommets, est  \textit{critique} \cite{Boniz,S.T} lorsque  tous ses sommets sont critiques.  De la caractérisation des tournois critiques, les auteurs de \cite{S.T} ont améliorés le corollaire~\ref{-1-2, pour tournoi} comme suit.
 \begin{theoreme} \cite{S.T} \label{ER -2, pour tournoi}
 Étant donné un tournoi indécomposable $T$ à au moins 7 sommets, il
existe deux sommets distincts $x$ et $y$ de $T$ tels que $T-\{x, y\}$ est indécomposable.
 \end{theoreme}
 Le théorème~\ref{ER -2, pour tournoi} est l'un des premiers résultats important sur les propriétés héréditaires descendantes de l'indécomposabilité. On en déduit que tout tournoi indécomposable et non critique d'ordre $n$, où $n \geqslant 5$, abrite un tournoi indécomposable de tout ordre compris entre 5 et $n$. 
 
 Dans \cite{BI}, les auteurs  ont donné un autre type d'hérédité qui s'énonce comme suit dans le cas des tournois.
\begin{corollaire} \cite{BI}\label{cor-1,2sommets}   \'Etant donné un tournoi indécomposable et 
non critique $T$, à au moins 7 sommets, il existe $x\in V(T)$ tel que le tournoi $T-x$ est aussi indécomposable et non critique.
\end{corollaire}
 Le théorème suivant peut être considéré comme un renforcement du théorème~\ref{ER +2, pour tournoi}, en ce qu'il permet d'agrandir la suite des sous-tournois indécomposables d'un tournoi indécomposable et non critique, d'un seul sommet au lieu de deux.
 \begin{theoreme} \cite{Gakuliu} \label{Gaku +1}
 Soit $T$ un tournoi indécomposable et non critique à au moins $6$ sommets et soit $H$ un sous-tournoi indécomposable de $T$ tel que $5\leqslant \vert H\vert <\vert T \vert$. Alors, $H$ s'abrite dans un sous-tournoi indécomposable à $(\vert H\vert +1)$ sommets de $T$.  
 \end{theoreme}
  Les résultats principaux de ce papier s'articulent autour du  théorème~\ref{Gaku +1}, il est organisé comme suit.  Dans la section~\ref{chapitre nouvelle preuve}, nous donnons une nouvelle preuve, plus courte et esthétique, du théorème~\ref{Gaku +1}. Dans la troisième section, nous proposons quelques applications  du théorème~\ref{Gaku +1}, en donnant des nouvelles preuves simples de Théorème~\ref{Houma} \cite{houma} et de Théorème~\ref{GL et D4} \cite{Gakuliu}.  Dans la section~\ref{chapitre sommets fortement critiques}, nous introduisons la notion de sommets fortement critiques dont on montre que leurs nombre ne dépasse pas 4 pour un tournoi indécomposable et non critique et que la borne 4 est optimale.
 
 \section{Nouvelle preuve du théorème~\ref{Gaku +1}}\label{chapitre nouvelle preuve}
 Dans cette section, nous proposons une nouvelle preuve  du théorème \ref{Gaku +1}, basée sur la notion de graphe d'indécomposabilité et du corollaire~\ref{cor-1,2sommets}. 
 
 Nous avons besoin d'introduire les notions suivantes. Un \textit{graphe} $G$ est la donnée d'un ensemble fini $V = V(G)$
de sommets avec un ensemble $ E = E(G)$ d'arêtes, où une arête est  une paire de sommets distincts. On note
$G =(V,E)$  et on dit que $G$ est défini sur $V$. Par exemple, pour tout entier $n \geqslant 3$, le chemin $P_n$ de longueur  
$n-1$ et le cycle $C_n$ de longueur $n$ sont les graphes définis sur $\{0,\ldots ,n-1\}$ de la façons suivante.
 Pour tous $i,j\in \{0,\ldots ,n-1\}$, $\{i,j\}$ est une arête de $P_n$ si $\vert i-j\vert =1$. Le cycle $C_n$ est
 alors obtenu à partir de $P_n$ en ajoutant l'arête $\{0, n-1\}$. Tout graphe isomorphe à $P_n$ (resp. $C_n$) est appelé \textit{chemin} (resp. \textit{cycle}).
 
  Soit $G=(V,E)$ un graphe. Un \textit{parcours} (de longueur $k-1$) de $G$ est une suite finie $(x_{1},\ldots,x_{k})$, où $k\geqslant 1$, de sommets de $G$ tels  que pour tout $i\in\lbrace 1,\ldots, k-1\rbrace$, on a $\lbrace x_{i},x_{i+1}\rbrace \in E$. On dit que ce parcours relie $x_{1} $ à $x_{k}$. Un \textit{parcours élémentaire} de $G$   est un parcours $(x_{1},\ldots,x_{k})$ de $G$ dont les sommets sont deux à deux distincts. Le graphe $G$ est \textit{connexe} lorsque pour tous  $x\neq y\in V$, il existe un parcours élémentaire de $G$ reliant $x$ à $y$. On définit sur $V$ la relation d'équivalence $R$ comme suit: pour tous $x,y\in V$, $x R y$ si et seulement s'il existe un parcours élémentaire de $G$ reliant $x$ à $y$. Les \textit{composantes connexes } du graphe $G$ sont les classes d'équivalence de $R$. Ainsi, une partie $C$ de $V$ est une composante connexe de $G$ si  $C$ est une partie  maximale (pour l'inclusion) de $V$, telle que $G[C]$ est connexe. Le graphe $G$ est donc connexe si et seulement s'il admet une seule composante connexe, et dans ce cas celle-ci est égale à $V$. Pour $x\in V$, l'ensemble des \textit{voisins}  de $x$ est $N_G(x)=\{y\in V: \{x,y\}\in E\}$. Le  sommet $x$ est  \textit{isolé} si $N_G(x)=\varnothing$.
 
 La notion de \textit{graphe d'indécomposablité} à été introduite par P. Ille \cite{I,imed.ill} de la façons
suivante. \`A chaque tournoi $T = (V,A)$ est associé son graphe d’indécomposabilité $I(T)$ défini sur $V$ comme suit. Pour tous $x\neq y\in V$, $\{x,y\}$ est une arête de $I(T)$ si $T-\{x,y \}$ est indécomposable. Ce graphe est un outil important dans dans l'étude de l'indécomposabilté critique comme le montre les deux résultats suivants.
\begin{lemme}~\cite{BI}\label{pi} Soit $T=(V,A)$ un tournoi indécomposable tel que $\vert V\vert \geqslant 5$. Pour tout sommet critique $x$ de $V$, $\vert N_{I(T)}(x)\vert \leqslant 2$. De plus,
\begin{enumerate}
\item si $N_{I(T)}(x)=\lbrace y\rbrace$, où $y\in V$, alors $V\setminus \lbrace x,y\rbrace$ est un intervalle de $T-x$;
\item si $N_{I(G)}(x)=\lbrace y,z\rbrace$, où $y\neq z \in V$, alors $\lbrace y,z\rbrace$ est un intervalle de $T- x$.
\end{enumerate}
\end{lemme}
\begin{corollaire}\cite{BI}\label{xxx}
 Soit $T$ un tournoi indécomposable et non critique à au moins $7$ sommets. Si $C$ est une composante connexe du graphe d'indécomposabilité $I(T)$ de $T$ telle que $\vert C \vert\geqslant 2$, alors $C$ contient un sommet non critique de $T$.
\end{corollaire}
Notons maintenant le fait suivant. 
\begin{Fait}\label{yyyy}
Tout tournoi $T$ à $2n$ sommets, où $n\geqslant 2$, abritant un tournoi transitif à $2n-1$ sommets, est décomposable.
\end{Fait}
\begin{proof}[\textbf{Preuve}]
On peut supposer que $V(T)=\lbrace 0,\ldots,2n-1\rbrace$ et que $V(T)-(2n-1)$ est l'ordre total usuel sur $\lbrace 0,\ldots,2n-2\rbrace$. Supposons que $T$ est indécomposable. Comme $\lbrace 1,\ldots,2n-1\rbrace$ n'est pas un intervalle de $T$, alors $2n-1\longrightarrow 0$. De plus, pour tout $i\in\lbrace 0,\ldots,2n-3\rbrace$, $T[\{2n-1,i\}]\not\simeq T[\{2n-1,i+1\}]$, car  autrement $\lbrace i,i+1\rbrace$ sera un intervalle non trivial du tournoi indécomposable $T$. Il s'ensuit que $(2n-1)\longrightarrow (2n-2)$. Ainsi, $\lbrace 0,\ldots,2n-3\rbrace \cup\lbrace 2n-1\rbrace $ est un intervalle non trivial de $T$. Contradiction.
\end{proof}
 Afin de rappeler la caractérisation des tournois critiques, nous introduisons, pour tout entier $n\geqslant 2$, les tournois $T_{2n+1}$, $U_{2n+1}$ et $W_{2n+1}$ définis sur $\{0, . . . , 2n\}$ comme suit. 

\begin{itemize}
\item $T_{2n+1}[\lbrace 0,\ldots ,n\rbrace]$ est l'ordre total usuel sur $\lbrace 0,\ldots ,n\rbrace$, $T_{2n+1}[\lbrace n+1,\ldots ,2n\rbrace]$ est l'ordre total usuel sur $\lbrace n+1,\ldots ,2n\rbrace$ et pour tout $i\in{\lbrace 0,\ldots,n-1\rbrace}$, $\lbrace i+1,\ldots,n\rbrace \longrightarrow i+n+~1 \longrightarrow \lbrace 0,\ldots ,i\rbrace$. (voir Figure \ref{T_{2n+1}});
\item  $U_{2n+1}$ est obtenu à partir de $T_{2n+1}$, en inversant tous les arcs $(i,j)\in A(T_{2n+1})$ tels que $n+1\leqslant i< j \leqslant 2n$. (voir Figure \ref{U_{2n+1}});
\item $W_{2n+1}[\lbrace 0,\ldots 2n-1\rbrace]$ est l'ordre total usuel sur $\lbrace 0,\ldots ,2n-1\rbrace$ et\\
 $\lbrace 2i-1:\; 1\leqslant i\leqslant n \rbrace\longrightarrow 2n\longrightarrow \lbrace 2i:\; 0\leqslant i\leqslant n-1\rbrace$. (voir Figure \ref{W_{2n+1}}).
\end{itemize}

\begin{figure}[h]
\begin{center}
\definecolor{ccwwff}{rgb}{0.8,0.4,1.}
\begin{tikzpicture}[line cap=round,line join=round,>=triangle 45,x=1cm,y=0.75cm]
\draw [line width=0.4pt,color=ccwwff] (2.,0.)-- (3.5,0.);
\draw [line width=0.4pt,color=ccwwff] (2.84,0.) -- (2.75,-0.12);
\draw [line width=0.4pt,color=ccwwff] (2.84,0.) -- (2.75,0.12);
\draw (2.,0.)-- (3.5,0.);
\draw [line width=0.4pt] (5.5,0.)-- (7.,0.);
\draw [line width=0.4pt] (6.34,0.) -- (6.25,-0.12);
\draw [line width=0.4pt] (6.34,0.) -- (6.25,0.12);
\draw [line width=0.4pt] (9.,0.)-- (10.5,0.);
\draw [line width=0.4pt] (9.84,0.) -- (9.75,-0.12);
\draw [line width=0.4pt] (9.84,0.) -- (9.75,0.12);
\draw (2.,0.)-- (3.5,0.);
\draw [line width=0.4pt] (3.5,0.)-- (4.44,0.);
\draw [line width=0.4pt] (4.06,0.) -- (3.97,-0.12);
\draw [line width=0.4pt] (4.06,0.) -- (3.97,0.12);
\draw [line width=0.4pt] (7.,0.)-- (7.86,0.);
\draw [line width=0.4pt] (7.52,0.) -- (7.43,-0.12);
\draw [line width=0.4pt] (7.52,0.) -- (7.43,0.12);
\draw (9.,0.)-- (10.5,0.);
\draw [line width=0.4pt] (10.5,0.)-- (9.8,2.);
\draw [line width=0.4pt] (10.12026846177447,1.0849472520729415) -- (10.263263002763923,1.039642050967372);
\draw [line width=0.4pt] (10.12026846177447,1.0849472520729415) -- (10.036736997236076,0.960357949032627);
\draw [line width=0.4pt] (2.84,1.98)-- (2.,0.);
\draw [line width=0.4pt] (2.3848505137829448,0.9071476396312261) -- (2.3095301861749693,1.0368659816227392);
\draw [line width=0.4pt] (2.3848505137829448,0.9071476396312261) -- (2.530469813825031,0.9431340183772586);
\draw [line width=0.4pt] (3.5,0.)-- (2.84,1.98);
\draw [line width=0.4pt] (3.1415395010584852,1.075381496824546) -- (3.2838419957660627,1.0279473319220205);
\draw [line width=0.4pt] (3.1415395010584852,1.075381496824546) -- (3.0561580042339385,0.9520526680779796);
\draw [line width=0.4pt] (4.64,0.4)-- (2.84,1.98);
\draw [line width=0.4pt] (3.6723612923671722,1.2493717544777034) -- (3.819162339303607,1.2801849435104367);
\draw [line width=0.4pt] (3.6723612923671722,1.2493717544777034) -- (3.660837660696394,1.0998150564895621);
\draw [line width=0.4pt] (6.26,2.)-- (5.5,0.);
\draw [line width=0.4pt] (5.848030401543879,0.9158694777470482) -- (5.7678259703293975,1.042626131274828);
\draw [line width=0.4pt] (5.848030401543879,0.9158694777470482) -- (5.992174029670602,0.9573738687251712);
\draw [line width=0.4pt] (7.,0.)-- (6.26,2.);
\draw [line width=0.4pt] (6.598769200638714,1.0844075658413095) -- (6.742543421121748,1.0416410658150455);
\draw [line width=0.4pt] (6.598769200638714,1.0844075658413095) -- (6.517456578878252,0.9583589341849524);
\draw [line width=0.4pt] (9.8,2.)-- (9.,0.);
\draw [line width=0.4pt] (9.366574839128132,0.9164370978203267) -- (9.288582797093772,1.0445668811624926);
\draw [line width=0.4pt] (9.366574839128132,0.9164370978203267) -- (9.511417202906234,0.9554331188375074);
\draw [line width=0.4pt] (9.8,2.)-- (8.24,0.38);
\draw [line width=0.4pt] (8.957572248164073,1.1251711807857663) -- (8.933561574381022,1.2732370024479036);
\draw [line width=0.4pt] (8.957572248164073,1.1251711807857663) -- (9.106438425618979,1.1067629975520954);
\draw [line width=0.4pt] (6.26,2.)-- (5.16,0.4);
\draw [line width=0.4pt] (5.659012405940164,1.1258362268220565) -- (5.611114969096074,1.2679834587464478);
\draw [line width=0.4pt] (5.659012405940164,1.1258362268220565) -- (5.808885030903923,1.132016541253552);
\draw [line width=0.4pt] (7.88,0.34)-- (6.26,2.);
\draw [line width=0.4pt] (7.007141156543065,1.2344109136657477) -- (7.155881218220998,1.2538117912759132);
\draw [line width=0.4pt] (7.007141156543065,1.2344109136657477) -- (6.984118781779,1.0861882087240855);
\draw [line width=0.4pt] (2.,-0.74)-- (3.,-0.74);
\draw [line width=0.4pt] (2.59,-0.74) -- (2.5,-0.86);
\draw [line width=0.4pt] (2.59,-0.74) -- (2.5,-0.62);
\draw [line width=0.4pt] (2.,-0.38)-- (2.,-0.74);
\draw [line width=0.4pt] (2.84,1.98)-- (4.,2.);
\draw [line width=0.4pt] (3.50998662604941,1.9915514935525762) -- (3.422068658070102,1.8700178319341207);
\draw [line width=0.4pt] (3.50998662604941,1.9915514935525762) -- (3.4179313419298984,2.1099821680658795);
\draw [line width=0.4pt] (5.2,2.)-- (6.26,2.);
\draw [line width=0.4pt] (5.82,2.) -- (5.73,1.88);
\draw [line width=0.4pt] (5.82,2.) -- (5.73,2.12);
\draw [line width=0.4pt] (6.26,2.)-- (7.28,1.98);
\draw [line width=0.4pt] (6.859982703949107,1.9882356332559) -- (6.767647511007867,1.870023061401193);
\draw [line width=0.4pt] (6.859982703949107,1.9882356332559) -- (6.772352488992135,2.109976938598807);
\draw [line width=0.4pt] (8.74,2.)-- (9.8,2.);
\draw [line width=0.4pt] (9.36,2.) -- (9.27,1.88);
\draw [line width=0.4pt] (9.36,2.) -- (9.27,2.12);
\draw [line width=0.4pt] (2.86,3.)-- (4.,3.);
\draw [line width=0.4pt] (3.52,3.) -- (3.43,2.88);
\draw [line width=0.4pt] (3.52,3.) -- (3.43,3.12);
\draw [line width=0.4pt] (2.86,2.75)-- (2.86,3.);
\begin{scriptsize}
\draw [fill=black] (2.,0.) circle (2.pt);
\draw[color=black] (1.9,-0.28) node {$0$};
\draw [fill=black] (3.5,0.) circle (2.pt);
\draw[color=black] (3.5,-0.28) node {$1$};
\draw [fill=black] (7.,0.) circle (2.pt);
\draw[color=black] (7,-0.28) node {$i+1$};
\draw [fill=black] (5.5,0.) circle (2.pt);
\draw[color=black] (5.5,-0.28) node {$i$};
\draw [fill=black] (9.,0.) circle (2.pt);
\draw[color=black] (9,-0.28) node {$n-1$};
\draw [fill=black] (10.5,0.) circle (2.0pt);
\draw[color=black] (10.5,-0.28) node {$n$};
\draw [fill=black] (4.76,0.) circle (0.5pt);
\draw [fill=black] (5.,0.) circle (0.5pt);
\draw [fill=black] (8.22,0.) circle (0.5pt);
\draw [fill=black] (8.54,0.) circle (0.5pt);
\draw [fill=black] (2.84,1.98) circle (2.pt);
\draw[color=black] (2.84,2.34) node {$n+1$};
\draw [fill=black] (6.26,2.) circle (2.pt);
\draw[color=black] (6.26,2.36) node {$i+n+1$};
\draw [fill=black] (9.8,2.) circle (2.pt);
\draw[color=black] (9.8,2.36) node {$2n$};
\draw [fill=black] (4.36,2.) circle (0.5pt);
\draw [fill=black] (4.8,2.) circle (0.5pt);
\draw [fill=black] (7.82,2.) circle (0.5pt);
\draw [fill=black] (8.36,2.) circle (0.5pt);
\end{scriptsize}
\end{tikzpicture}
\vspace{0.2cm} \caption{\textbf{Le Tournoi critique $T_{2n+1}$}}
\label{T_{2n+1}}
\end{center}
\end{figure}
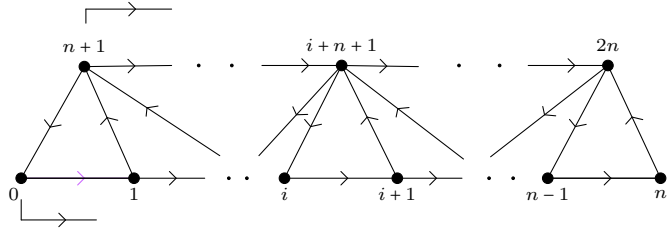
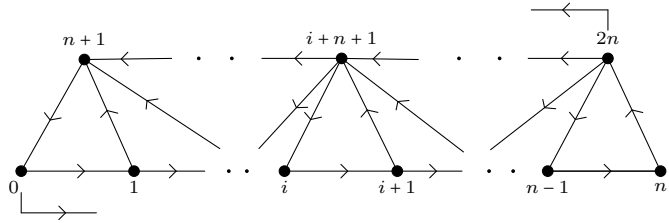
\begin{figure}[h]
\begin{center}
\begin{tikzpicture}[line cap=round,line join=round,>=triangle 45,x=1cm,y=0.75cm]
\draw [line width=0.4pt] (5.5,0.)-- (7.,0.);
\draw [line width=0.4pt] (6.34,0.) -- (6.25,-0.12);
\draw [line width=0.4pt] (6.34,0.) -- (6.25,0.12);
\draw [line width=0.4pt] (9.,0.)-- (10.5,0.);
\draw [line width=0.4pt] (9.84,0.) -- (9.75,-0.12);
\draw [line width=0.4pt] (9.84,0.) -- (9.75,0.12);
\draw [line width=0.4pt] (3.5,0.)-- (4.44,0.);
\draw [line width=0.4pt] (4.06,0.) -- (3.97,-0.12);
\draw [line width=0.4pt] (4.06,0.) -- (3.97,0.12);
\draw [line width=0.4pt] (7.,0.)-- (7.86,0.);
\draw [line width=0.4pt] (7.52,0.) -- (7.43,-0.12);
\draw [line width=0.4pt] (7.52,0.) -- (7.43,0.12);
\draw (9.,0.)-- (10.5,0.);
\draw [line width=0.4pt] (10.5,0.)-- (9.8,2.);
\draw [line width=0.4pt] (10.120268461774472,1.0849472520729413) -- (10.263263002763923,1.039642050967373);
\draw [line width=0.4pt] (10.120268461774472,1.0849472520729413) -- (10.03673699723608,0.9603579490326266);
\draw [line width=0.4pt] (2.84,1.98)-- (2.,0.);
\draw [line width=0.4pt] (2.384850513782944,0.9071476396312258) -- (2.309530186174969,1.036865981622739);
\draw [line width=0.4pt] (2.384850513782944,0.9071476396312258) -- (2.5304698138250306,0.9431340183772583);
\draw [line width=0.4pt] (3.5,0.)-- (2.84,1.98);
\draw [line width=0.4pt] (3.141539501058485,1.075381496824546) -- (3.283841995766062,1.0279473319220203);
\draw [line width=0.4pt] (3.141539501058485,1.075381496824546) -- (3.056158004233938,0.9520526680779793);
\draw [line width=0.4pt] (4.64,0.4)-- (2.84,1.98);
\draw [line width=0.4pt] (3.6723612923671727,1.2493717544777034) -- (3.8191623393036074,1.2801849435104367);
\draw [line width=0.4pt] (3.6723612923671727,1.2493717544777034) -- (3.6608376606963953,1.099815056489562);
\draw [line width=0.4pt] (6.26,2.)-- (5.5,0.);
\draw [line width=0.4pt] (5.8480304015438795,0.9158694777470479) -- (5.767825970329398,1.0426261312748275);
\draw [line width=0.4pt] (5.8480304015438795,0.9158694777470479) -- (5.992174029670603,0.9573738687251708);
\draw [line width=0.4pt] (7.,0.)-- (6.26,2.);
\draw [line width=0.4pt] (6.598769200638716,1.0844075658413093) -- (6.7425434211217485,1.0416410658150452);
\draw [line width=0.4pt] (6.598769200638716,1.0844075658413093) -- (6.517456578878254,0.9583589341849521);
\draw [line width=0.4pt] (9.8,2.)-- (9.,0.);
\draw [line width=0.4pt] (9.366574839128134,0.9164370978203263) -- (9.288582797093772,1.0445668811624924);
\draw [line width=0.4pt] (9.366574839128134,0.9164370978203263) -- (9.511417202906234,0.9554331188375071);
\draw [line width=0.4pt] (9.8,2.)-- (8.24,0.38);
\draw [line width=0.4pt] (8.957572248164073,1.125171180785766) -- (8.933561574381024,1.2732370024479034);
\draw [line width=0.4pt] (8.957572248164073,1.125171180785766) -- (9.10643842561898,1.1067629975520952);
\draw [line width=0.4pt] (6.26,2.)-- (5.16,0.4);
\draw [line width=0.4pt] (5.659012405940165,1.1258362268220565) -- (5.611114969096076,1.2679834587464478);
\draw [line width=0.4pt] (5.659012405940165,1.1258362268220565) -- (5.808885030903924,1.132016541253552);
\draw [line width=0.4pt] (7.88,0.34)-- (6.26,2.);
\draw [line width=0.4pt] (7.0071411565430655,1.2344109136657475) -- (7.155881218220999,1.2538117912759132);
\draw [line width=0.4pt] (7.0071411565430655,1.2344109136657475) -- (6.984118781779004,1.0861882087240853);
\draw [line width=0.4pt] (2.,-0.74)-- (3.,-0.74);
\draw [line width=0.4pt] (2.59,-0.74) -- (2.5,-0.86);
\draw [line width=0.4pt] (2.59,-0.74) -- (2.5,-0.62);
\draw [line width=0.4pt] (2.,-0.38)-- (2.,-0.74);
\draw [line width=0.4pt] (2.,0.)-- (3.5,0.);
\draw [line width=0.4pt] (2.84,0.) -- (2.75,-0.12);
\draw [line width=0.4pt] (2.84,0.) -- (2.75,0.12);
\draw [line width=0.4pt] (4.,2.)-- (2.84,1.98);
\draw [line width=0.4pt] (3.3300133739505906,1.988448506447424) -- (3.4179313419299002,2.10998216806588);
\draw [line width=0.4pt] (3.3300133739505906,1.988448506447424) -- (3.4220686580701023,1.87001783193412);
\draw [line width=0.4pt] (6.26,2.)-- (5.2,2.);
\draw [line width=0.4pt] (5.64,2.) -- (5.73,2.12);
\draw [line width=0.4pt] (5.64,2.) -- (5.73,1.88);
\draw [line width=0.4pt] (7.28,1.98)-- (6.26,2.);
\draw [line width=0.4pt] (6.680017296050895,1.9917643667440992) -- (6.7723524889921345,2.1099769385988063);
\draw [line width=0.4pt] (6.680017296050895,1.9917643667440992) -- (6.767647511007866,1.870023061401192);
\draw [line width=0.4pt] (9.8,2.)-- (8.74,2.);
\draw [line width=0.4pt] (9.18,2.) -- (9.27,2.12);
\draw [line width=0.4pt] (9.18,2.) -- (9.27,1.88);
\draw [line width=0.4pt] (9.8,2.86)-- (8.78,2.86);
\draw [line width=0.4pt] (9.2,2.86) -- (9.29,2.98);
\draw [line width=0.4pt] (9.2,2.86) -- (9.29,2.74);
\draw [line width=0.4pt] (9.8,2.5)-- (9.8,2.86);
\begin{scriptsize}
\draw [fill=black] (2.,0.) circle (2.pt);
\draw[color=black] (1.9,-0.28) node {$0$};
\draw [fill=black] (3.5,0.) circle (2.pt);
\draw[color=black] (3.5,-0.28) node {$1$};
\draw [fill=black] (7.,0.) circle (2.pt);
\draw[color=black] (7.,-0.28) node {$i+1$};
\draw [fill=black] (5.5,0.) circle (2.pt);
\draw[color=black] (5.5,-0.28) node {$i$};
\draw [fill=black] (9.,0.) circle (2.pt);
\draw[color=black] (9.,-0.28) node {$n-1$};
\draw [fill=black] (10.5,0.) circle (2.0pt);
\draw[color=black] (10.5,-0.28) node {$n$};
\draw [fill=black] (4.76,0.) circle (0.5pt);
\draw [fill=black] (5.,0.) circle (0.5pt);
\draw [fill=black] (8.22,0.) circle (0.5pt);
\draw [fill=black] (8.54,0.) circle (0.5pt);
\draw [fill=black] (2.84,1.98) circle (2.pt);
\draw[color=black] (2.84,2.34) node {$n+1$};
\draw [fill=black] (6.26,2.) circle (2.pt);
\draw[color=black] (6.26,2.36) node {$i+n+1$};
\draw [fill=black] (9.8,2.) circle (2.pt);
\draw[color=black] (9.8,2.36) node {$2n$};
\draw [fill=black] (4.36,2.) circle (0.5pt);
\draw [fill=black] (4.8,2.) circle (0.5pt);
\draw [fill=black] (7.82,2.) circle (0.5pt);
\draw [fill=black] (8.36,2.) circle (0.5pt);
\end{scriptsize}
\end{tikzpicture}
\vspace{0.2cm} \caption{\textbf{Le Tournoi critique $U_{2n+1}$}}
\label{U_{2n+1}}
\end{center}
\end{figure}~
~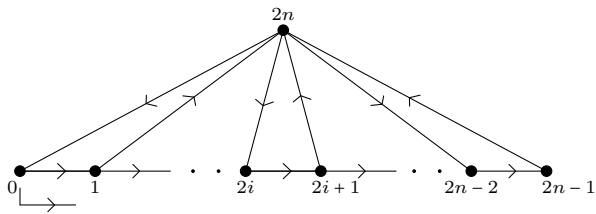
\begin{figure}[h]
\begin{center}
\begin{tikzpicture}[line cap=round,line join=round,>=triangle 45,x=1.cm,y=0.75cm]
\draw [line width=0.4pt] (3.,0.)-- (4.,0.);
\draw [line width=0.4pt] (6.,0.)-- (7.,0.);
\draw [line width=0.4pt] (9.,0.)-- (10.,0.);
\draw [line width=0.4pt] (9.59,0.) -- (9.5,-0.12);
\draw [line width=0.4pt] (9.59,0.) -- (9.5,0.12);
\draw (6.,0.)-- (7.,0.);
\draw (6.605,0.) -- (6.5,-0.135);
\draw (6.605,0.) -- (6.5,0.135);
\draw (3.,0.)-- (4.,0.);
\draw (3.605,0.) -- (3.5,-0.135);
\draw (3.605,0.) -- (3.5,0.135);
\draw [line width=0.4pt] (6.5,2.5)-- (3.,0.);
\draw [line width=0.4pt] (4.676763987591394,1.1976885625652811) -- (4.680251416753708,1.3476480165448081);
\draw [line width=0.4pt] (4.676763987591394,1.1976885625652811) -- (4.819748583246292,1.1523519834551923);
\draw [line width=0.4pt] (4.,0.)-- (6.5,2.5);
\draw [line width=0.4pt] (5.31363961030679,1.313639610306789) -- (5.334852813742386,1.1651471862576142);
\draw [line width=0.4pt] (5.31363961030679,1.313639610306789) -- (5.165147186257615,1.3348528137423856);
\draw [line width=0.4pt] (6.5,2.5)-- (6.,0.);
\draw [line width=0.4pt] (6.232349547837564,1.1617477391878168) -- (6.132330318917091,1.2735339362165814);
\draw [line width=0.4pt] (6.232349547837564,1.1617477391878168) -- (6.367669681082912,1.2264660637834175);
\draw [line width=0.4pt] (7.,0.)-- (6.5,2.5);
\draw [line width=0.4pt] (6.732349547837564,1.3382522608121825) -- (6.867669681082912,1.2735339362165818);
\draw [line width=0.4pt] (6.732349547837564,1.3382522608121825) -- (6.632330318917091,1.226466063783418);
\draw [line width=0.4pt] (10.,0.)-- (6.5,2.5);
\draw [line width=0.4pt] (8.176763987591395,1.3023114374347187) -- (8.319748583246291,1.3476480165448081);
\draw [line width=0.4pt] (8.176763987591395,1.3023114374347187) -- (8.180251416753709,1.1523519834551923);
\draw [line width=0.4pt] (6.5,2.5)-- (9.,0.);
\draw [line width=0.4pt] (7.813639610306788,1.1863603896932107) -- (7.665147186257614,1.1651471862576142);
\draw [line width=0.4pt] (7.813639610306788,1.1863603896932107) -- (7.834852813742384,1.3348528137423856);
\draw [line width=0.4pt] (4.,0.)-- (5.,0.);
\draw [line width=0.4pt] (4.59,0.) -- (4.5,-0.12);
\draw [line width=0.4pt] (4.59,0.) -- (4.5,0.12);
\draw [line width=0.4pt] (7.,0.)-- (8.,0.);
\draw [line width=0.4pt] (7.59,0.) -- (7.5,-0.12);
\draw [line width=0.4pt] (7.59,0.) -- (7.5,0.12);
\draw [line width=0.4pt] (3.,-0.6)-- (3.74,-0.6);
\draw [line width=0.4pt] (3.46,-0.6) -- (3.37,-0.72);
\draw [line width=0.4pt] (3.46,-0.6) -- (3.37,-0.48);
\draw [line width=0.4pt] (3.,-0.6)-- (3.,-0.3);
\begin{scriptsize}
\draw [fill=black] (3.,0.) circle (2pt);
\draw[color=black] (2.9,-0.28) node {$0$};
\draw [fill=black] (4.,0.) circle (2pt);
\draw[color=black] (4.,-0.28) node {$1$};
\draw [fill=black] (6.,0.) circle (2pt);
\draw[color=black] (6.,-0.28) node {$2i$};
\draw [fill=black] (7.,0.) circle (2pt);
\draw[color=black] (7.2,-0.28) node {$2i+1$};
\draw [fill=black] (9.,0.) circle (2pt);
\draw[color=black] (9.,-0.28) node {$2n-2$};
\draw [fill=black] (10.,0.) circle (2pt);
\draw[color=black] (10.3,-0.28) node {$2n-1$};
\draw [fill=black] (6.5,2.5) circle (2pt);
\draw[color=black] (6.5,2.78) node {$2n$};
\draw [fill=black] (5.3,0.) circle (0.5pt);
\draw [fill=black] (5.64,0.) circle (0.5pt);
\draw [fill=black] (8.24,0.) circle (0.5pt);
\draw [fill=black] (8.58,0.) circle (0.5pt);
\end{scriptsize}
\end{tikzpicture}  
\vspace{0.2cm} \caption{\textbf{Le Tournoi critique $W_{2n+1}$}}
\label{W_{2n+1}}
\end{center}
\end{figure}

Les digraphes critiques ont été caractérisés indépendamment par Schmerl et Trotter en 1993 \cite{S.T}, et par Bonizonni en 
1994 \cite{Boniz}. En 2009,  Boudabbous et Ille \cite{BI} ont retrouvé cette caractérisation en utilisant   la notion de graphe d'indécomposabilité. Les auteurs de \cite{BI} ont trouvé que pour chaque entier $n \geqslant 5 $, il existe exactement 5 ou 6 digraphes critiques d'ordre $n$, suivant que $n$ impair ou pair respectivement. Ces résultats s'énoncent comme suit dans le cas des tournois.
\begin{theoreme}\cite{BI}\label{g.indec des tournois crit}
\`A isomorphisme près, les tournois critiques sont $T_{2n+1}$, $U_{2n+1}$ et $W_{2n+1}$ où $n\geqslant 2$. De plus, en considérant la permutation\\
 $f_{2n+1}:x\mapsto (n+1)x $ de $\mathbb{Z}_{2n+1}$, les assertions suivantes sont vérifiées. 
\begin{enumerate}
\item  $I(T_{2n+1}) $ est un cycle. Plus précisément, $I(T_{2n+1})=f_{2n+1}(C_{2n+1})$;
\item $I(U_{2n+1}) $ est un chemin. Plus précisément, $I(U_{2n+1}) =f_{2n+1}(P_{2n+1})$;
\item $I(W_{2n+1}) -2n = P_{2n}  $ et $2n$ est un sommet isolé de  $I(W_{2n+1})$.
\end{enumerate}
\end{theoreme}
Nous ajoutons le lemme suivant.
 \begin{lemme}\label{T abrite T}
 Soit un  entier $n\geqslant2$. \`A isomorphisme près, les sous-tournois indécomposables de $T_{2n+1}$ (resp. $U_{2n+1}$, $W_{2n+1}$), à au moins $5$ sommets, sont les tournois $T_{2m+1}$ (resp. $U_{2m+1}$, $W_{2m+1}$ ), où $2\leqslant m\leqslant n$.
 \end{lemme}  
\begin{proof}[\textbf{Preuve}]
Considérons un tournoi non critique $D_{2n+1}=T_{2n+1},U_{2n+1}$ ou $W_{2n+1}$, où $n\geqslant 3$. Il suffit de vérifier que pour toute arête $X $ du graphe d'indécomposabilité $I(D_{2n+1})$ de $D_{2n+1}$, si $D_{2n+1}=T_{2n+1}$ (resp. $U_{2n+1} , W_{2n+1}$ ), alors $D_{2n+1}-X\simeq T _{2n-1}$ (resp. $U_{2n-1} , W_{2n-1}$ ).\\
Supposons d'abord que $D_{2n+1} = W_{2n+1}$. D'après le théorème~\ref{g.indec des tournois crit}, $X=\lbrace i, i+1\rbrace$, où $i\in\lbrace 0,\ldots, 2n-2\rbrace$. L'application :  \begin{equation*}
\begin{array}{lrcl}
 &\lbrace 0,\ldots, 2n\rbrace\setminus\lbrace i,i+1\rbrace &\longrightarrow &\lbrace 0,\ldots, 2n-2\rbrace  \\  &k&\longmapsto &
$$ \left \lbrace\begin{array}{ll} k & \text{si} \; \; 0\leqslant k \leqslant i-1 \\
k-2 &\text{si}\; \;  i+2\leqslant k \leqslant 2n, 
\end{array} \right.$$    
\end{array}
\end{equation*}
est un isomorphisme de $W_{2n+1}- X$ sur $W_{2n-1}$.\\
Supposons maintenant que $D_{2n+1}= T_{2n+1}$. D'après le théorème \ref{g.indec des tournois crit}, $X =\lbrace i,i~+~n~+~1~\rbrace$, où $i\in\lbrace 0,\ldots, 2n\rbrace$. Comme la permutation de $\mathbb{Z}_{2n+1}$, définie par $i\longmapsto i+1$ est un automorphisme de $T_{2n+1}$, on peut supposer que $i=0$, c'est à dire  $X=\lbrace 0,n+1\rbrace$. L'application : 
\begin{equation*}
\begin{array}{lrcl}
 &\lbrace 0,\ldots, 2n\rbrace\setminus\lbrace 0,n+1\rbrace &\longrightarrow &\lbrace 0,\ldots, 2n-2\rbrace  \\  &k&\longmapsto &
$$ \left \lbrace\begin{array}{ll} k-1 & \text{si} \; \; 1\leqslant k \leqslant n \\
k-2 &\text{si}\; \;  n+2\leqslant k \leqslant 2n, 
\end{array} \right.$$    
\end{array}
\end{equation*}
est un isomorphisme de $T_{2n+1}-X$ sur $T_{2n-1}$.\\
Supposons enfin, que $D_{2n+1}=U_{2n+1}$. D'après le théorème \ref{g.indec des tournois crit}, $X=\lbrace i,n+i\rbrace$ ou $\lbrace i-1,n+i\rbrace$, où $i\in\lbrace 1,\ldots,n\rbrace$. Dans le cas où $X=\lbrace i,n+i\rbrace$, l'application :  
\begin{equation*}
\begin{array}{lrcl}
 &\lbrace 0,\ldots, 2n\rbrace\setminus\lbrace i,n+i\rbrace &\longrightarrow &\lbrace 0,\ldots, 2n-2\rbrace  \\  &k&\longmapsto &$$ \left \lbrace\begin{array}{lll} k & \text{si} \; \; 0\leqslant k \leqslant i-1 \\
k-1 &\text{si}\; \;  i+1\leqslant k \leqslant i+n-1 \\
k-2 &\text{si}\; \;  i+n+1\leqslant k \leqslant 2n,
\end{array} \right.$$    
\end{array}
\end{equation*}
est un isomorphisme de $U_{2n+1}-\lbrace i,n+i\rbrace$ sur $U_{2n-1}$. Dans le cas où $X=\lbrace i-1,n+i\rbrace$, l'application : 
\begin{equation*}
\begin{array}{lrcl}
 &\lbrace 0,\ldots, 2n\rbrace\setminus\lbrace i-1,n+i\rbrace &\longrightarrow &\lbrace 0,\ldots, 2n-2\rbrace  \\  &k&\longmapsto &$$ \left \lbrace\begin{array}{lll} k & \text{si} \; \; 0\leqslant k \leqslant i-2 \\
k-1 &\text{si}\; \;  i\leqslant k \leqslant i+n-1 \\
k-2 &\text{si}\; \;  i+n+1\leqslant k \leqslant 2n,
\end{array} \right.$$    
\end{array}
\end{equation*}
est un isomorphisme de $U_{2n+1}-\lbrace i-1,n+i\rbrace$ sur $U_{2n-1}$
~\end{proof}
Disons alors, que deux tournois critiques sont de même type, lorsque l'un s'abrite dans l'autre. Ainsi, d'après le lemme \ref{T abrite T}, deux tournois critiques $T$ et $T'$ sont de même type, lorsqu'il existe deux entiers $m$,$n\geqslant 2$ tels que l'une des assertions suivantes est vérifiée. 
\begin{itemize}
\item $T\simeq T_{2n+1}$ et $T'\simeq T_{2m+1}$;
\item $T\simeq U_{2n+1}$ et $T'\simeq U_{2m+1}$;
\item $T\simeq W_{2n+1}$ et $T'\simeq W_{2m+1}$.
\end{itemize}

\begin{proof}[\textbf{Preuve du théorème~\ref{Gaku +1}}]
On se fixe un tournoi indécomposable $H$ à au moins $5$ sommets. Nous montrons que pour tout tournoi indécomposable et non critique $T$ à au moins $\mid\! H\!\mid +1$ sommets, si $H$ s'abrite dans $T$, alors $H$ s'abrite dans un sous-tournoi indécomposable à $(\mid\! H\!\mid +1)$ sommets de $T$. Nous raisonnons par récurrence sur $\mid\! T\!\mid$. Le résultat est évident si $\mid\! T\!\mid=\mid\! H\!\mid +1$. Considérons un tournoi indécomposable et non critique $T$ qui abrite $H$ et tel que  $\mid\! T\!\mid\geqslant \mid\! H\!\mid+2$. Soit $X\subset V(T)$ tel que $T[X]\simeq H$. Notons le fait suivant: \textit{s'il existe $x\in V(T)\setminus X$ tel que $T-x$ est indécomposable et non critique, la preuve est achevée en appliquant l'hypothèse de récurrence à $T-x$}. Nous discutons suivant la parité de $\mid\! V(T)\setminus X\!\mid$.\\
 Supposons d'abord que $\mid\! V(T)\setminus X\!\mid$ est impair. On a donc  $\mid\! V(T)\setminus X\!\mid\geqslant 3$. En appliquant le théorème~\ref{ER +2, pour tournoi} plusieur fois à partir de $T[X]$, on obtient un sommet $z\in V(T)\setminus X$, tel que $T-z$ est indécomposable. D'après le fait ci-dessus, on peut supposer que $T-z$ est critique. Comme les tournois critique sont d'ordres impairs, alors   
 \begin{equation} 
 \mid\! V(T)\!\mid \text{est pair.}
\label{v pair}
\end{equation}  D'après le lemme~\ref{T abrite T}, $T[X]=(T-z)[X]$ est critique et de même type que le tournoi critique $T-z$ à au moins $7$ sommets.  D'après le corollaire~\ref{cor-1,2sommets}, il existe $u\neq v\in V(T)$ tels que $T-u$ et $T-\lbrace u,v\rbrace $ sont indécomposables. Comme $T-z$ est critique, $z\notin\lbrace u,v\rbrace$. Pour une contradiction, supposons que $u$ est un sommet isolé du graphe d'indécomposabilité $I(T-z)$ de $T-z$. D'après le théorème~\ref{g.indec des tournois crit}, $T-z =W_{2n+1}$, où $n\geqslant 3$, et $T-\lbrace u,z\rbrace$ est un ordre total. Ainsi, $T-\lbrace u,v,z\rbrace$ est aussi un ordre total. D'après (\ref{v pair}), $\mid\! V(T)-\lbrace u,v,z\rbrace\!\mid$  est impair. Donc, d'après le fait~\ref{yyyy}, $T-\lbrace u,v\rbrace$ est décomposable. Contradiction. Ainsi, il existe un sommet $t\in V(T)\setminus \lbrace u,z\rbrace$ tel que $T-\lbrace z,u,t\rbrace$ est  indécomposable et donc critique d'après le lemme~\ref{T abrite T}. Comme $T[X]$ est critique de même type que $T-\lbrace z,u,t\rbrace$ et $\mid\! X\!\mid\;  \leqslant\; \;  \mid\!T\!\mid~-3$, alors  $T[X]$, et donc $H$,  s'abrite dans $T-\lbrace z,u,t\rbrace$. En particulier, $H$ s'abrite dans le tournoi indécomposable et non critique $T-u$.  Nous concluons alors, en appliquant l'hypothèse de récurrence au tournoi indécomposable et non critique $T-u$.\\
  Supposons maintenant que $\mid\! V(T)\setminus X\!\mid$ est pair. Comme $\mid\! V(T)\setminus X\!\mid \geqslant 2$, alors, en appliquant le théorème~\ref{ER +2, pour tournoi} à partir de $T[X]$, on obtient deux sommets $a\neq b \in V(T)\setminus X$ tel que $T-\lbrace a,b\rbrace$ est indécomposable. Nous pouvons supposer que $T-a$ et $T-b$ sont décomposables, puisqu' autrement, il suffit d'appliquer l'hypothèse de récurrence à $T-a$ ou $T-b$. Considérons la composante connexe $C$ de $I(T)$ contenant $\lbrace a,b\rbrace$. D'après le corollaire~\ref{xxx}, $C$ contient un sommet non critique de $T$. Maintenant, parmi les parcours élémentaires de $I(G)$, reliant un sommet $c\in\lbrace a,b\rbrace$ à un sommet non critique de $T$, choisissons un parcours  $P=( a_{1}=c,\ldots,a_{k})$, où $k\geqslant 2$, de longueur $k$ minimum. Le sommet $a_{k}$ est alors non critique. On peut supposer que $c=b$. Par minimalité de la longueur $k$ de $P$, $a\notin\lbrace a_{1},\ldots,a_{k-1}\rbrace$ et $a_{1},\ldots,a_{k-1}$ sont des sommets critiques de $T$. Considérons le parcours élémentaire $Q=( a_{0}~=a,a_{1}~=~b,\ldots,a_{k})$ et soit $i\in\lbrace 1,\ldots,k-1\rbrace$. D'après le lemme \ref{pi}, $\lbrace a_{i-1},a_{i+1}\rbrace$ est un intervalle de $T-a_{i}$. Donc $T-\lbrace a_{i},a_{i-1}\rbrace \simeq T-\lbrace a_{i},a_{i+1}\rbrace$. En faisant varier $i$, on obtient $T-\lbrace a,b\rbrace =T-\lbrace a_{0},a_{1}\rbrace\simeq T-\lbrace a_{k-1},a_{k}\rbrace$. Comme $H$ s'abrite dans $T-\lbrace a,b\rbrace$, alors $H$ s'abrite dans $T-\lbrace a_{k-1},a_{k}\rbrace$ et donc dans $T-a_{k}$. Nous pouvons alors appliquer l'hypothèse de récurrence au tournoi indécomposable et non critique $T-a_{k}$.
\end{proof}
\section{Quelques applications du théorème~\ref{Gaku +1}}\label{chapitre application de Gaku}
Comme tout tournoi indécomposable abrite un tournoi indécomposable à trois  sommets \cite{er} et comme $T_5, U_5$ et $W_5$ sont les seuls tournois indécomposables à 5 sommets, alors en appliquant le théorème~\ref{ER +2, pour tournoi} nous obtenons le fait  suivant.
\begin{Fait}\label{zzz}
Tout tournoi indécomposable à au moins 5 sommets, abrite $T_5, U_5$ ou 

$W_5$.
\end{Fait}
Le fait~\ref{zzz} implique que les tournois indécomposables à au moins $5$ sommets, peuvent être étudiés suivant les tournois indécomposables à $5$ sommets qu'ils abritent. Nous proposons dans cette section, deux résultats avec des nouvelles preuves fournies par le théorème~\ref{Gaku +1}. 

\subsection{\textbf{ Tournois indécomposables à 5 sommets dans un tournoi indécomposable}}
En 2003, B.J. Latka \cite{Latka} a caractérisé les tournois indécomposables n'abritant pas $W_{5}$.  Afin de rappeler cette caractérisation, nous introduisons le tournoi de Paley $\mathscr{P}_{7}$ \index{$\mathscr{P}_{7}$  (tournoi de Paley)} défini sur $\mathbb{Z}_{7}$ par $A(\mathscr{P}_{7})=\lbrace (i,j) :\;  j-i\in\lbrace 1,2,4\rbrace\rbrace$. Notons que pour tous $x,y\in\mathbb{Z}_{7}$, $\mathscr{P}_{7}-x\simeq \mathscr{P}_{7}-y$ et posons $B_{6}= \mathscr{P}_{7}-6$. 
\begin{theoreme}\cite{Latka}\label{Latka}
\`A isomorphisme près, les tournois indécomposables à au moins $5$ sommets et omettant $W_{5}$ sont les tournois $B_{6}$, $\mathscr{P}_{7}$, $T_{2n+1}$ et $U_{2n+1}$, où $n\geqslant 2$.
\end{theoreme}
 En 2006, H. Belkhechine et I. Boudabbous retrouvent une nouvelle preuve du théorème \ref{Latka} comme conséquence du théorème \ref{Houma} \cite{houma} énoncé ci-dessous, dont nous donnons à notre tour une nouvelle preuve utilisant le théorème \ref{Gaku +1}. 

\begin{theoreme} \cite{houma}\label{Houma}
Si un tournoi indécomposable $T$ abrite $T_{5}$, alors 

$T\in\lbrace T_{2n+1}: n\geqslant 2\rbrace$  ou $T$ abrite $U_{5}$ et $W_{5}$. 
\end{theoreme}
\begin{proof}[\textbf{Preuve}]
Soit $T$ un tournoi indécomposable abritant $T_{5}$. Supposons que $T\notin\mathcal{T}=\lbrace T_{2n+1}; n\geqslant 2\rbrace$. Comme seul $T_5, U_5$ et $W_5$ sont les tournois indécomposables à 5 sommets, alors $\mid\! T\!\mid\geqslant 6$. de plus, d'après le lemme \ref{T abrite T}, $T\notin\mathcal{U}=\lbrace U_{2n+1}: n\geqslant 2\rbrace$ et $T\notin\mathcal{W}=\lbrace W_{2n+1}: n\geqslant 2\rbrace$. Le théorème \ref{Gaku +1}, assure l'existence d'un sous-tournoi indécomposable $H$ à $6$ sommets de $T$ tel que $H$ abrite $T_{5}$. On peut supposer que $V(H)=\lbrace 0,1,2,3,4,u\rbrace$ et que $T_{5}$ est un sous-tournoi de $H$. Quitte à remplacer $T$ par ${T}^{\star}$ (puisque $T_{5}$, $U_{5}$ et $W_{5}$ sont autoduaux), on peut supposer que $\mid\! N^{-}_{H}(u)\!\mid\leqslant 2$. Si $\mid\! N^{-}_{H}(u)\!\mid=0$, alors $\lbrace 0,1,2,3,4\rbrace$ est un intervalle non trivial de $H$. Ceci contredit le fait que $H$ est indécomposable. Si $\mid\! N^{-}_{H}(u)\!\mid=1$, sans perte de généralité, on peut supposer que $0\longrightarrow u$. Alors, $u\longrightarrow\lbrace 1,2,3,4\rbrace$ et on a $T[\lbrace 3,4,0,1,u\rbrace]\simeq U_{5}$ et $T[\lbrace 0,u,1,2,4\rbrace]\simeq W_{5}$. D'où $H$ (et donc $T$) abrite $U_{5}$ et $W_{5}$. Enfin, supposons que $\mid\! N^{-}_{H}(u)\!\mid=2$. Alors, il existe $i\in\lbrace 0,1,2,3,4\rbrace$ tel que $ N^{-}_{H}(u)=\lbrace i,i+1\rbrace$ ou $ N^{-}_{H}(u)=\lbrace i,i+2\rbrace$   ($i+1$ et $i+2$ sont pris modulo $5$). Dans le cas où $ N^{-}_{H}(u)=\lbrace i,i+1\rbrace$, sans perte de généralité, on peut prendre $i=0$. Donc $\lbrace 0,1\rbrace\longrightarrow u$ et $u\longrightarrow\lbrace 2,3,4\rbrace$. Alors,  $\lbrace u,2\rbrace$ est un intervalle non trivial de $H$. Ceci contredit l'indécomposabilité de $H$. Il s'ensuit que $N^{-}_{H}(u)=\lbrace i,i+2\rbrace$. Sans perte de généralité, on peut supposer que $\lbrace 0,2\rbrace\longrightarrow u\longrightarrow\lbrace 1,3,4\rbrace$. Dans ce cas, on a $T[\lbrace 2,u,4,1,0\rbrace]\simeq U_{5}$ et $T[\lbrace 2,u,3,4,1\rbrace]\simeq W_{5}$.
\end{proof}
\begin{corollaire}\label{T,uw5 free}
Soit $T$ un tournoi indécomposable à au moins $5$ sommets. Le tournoi $T$ omet $U_5$ et $W_5$ si et seulement si $T\simeq T_{2n+1}$ pour un certain entier~ ~$n\geqslant 2$.
\end{corollaire}
\begin{proof}[\textbf{Preuve}]
 Si $T\simeq T_{2n+1}$ pour un certain entier $n\geqslant 2$, alors, d'après le lemme \ref{T abrite T}, tous les sous-tournois à $5$ sommets  de $T$ sont isomorphes à $T_5$. Réciproquement, Soit $T$ un tournoi indécomposable qui omet $U_5$ et $W_5$. D'après le fait~\ref{zzz}, $T$ abrite $T_5$. Il en résulte du théorème \ref{Houma} que $T\simeq T_{2n+1}$.
\end{proof}
D'après le théorème \ref{Houma}, remarquons que tout tournoi indécomposable $T\notin\lbrace T_{2n+1} : n\geqslant 2\rbrace$ qui n'abrite pas $U_{5}$, n'abrite pas aussi $T_{5}$.

\subsection{ \textbf{tournois omettant le diamant $D_{4}$}}
 Le théorème suivant (Gnanvo et Ille \cite{G-Ille}, Lopez et Rauzy \cite{L-R}) donne une caractérisation des tournois indécomposables omettant les tournois $D_{4}$ et ${D_{4}}^{\star}$.
\begin{theoreme} \cite{G-Ille, L-R}
Soit $T$ un tournoi indécomposable à au moins $5$ sommets. Le tournoi $T$ omet $D_{4}$ et ${D_{4}}^{\star}$ si et seulement si $T\simeq T_{2n+1}$ pour un certain entier $n\geqslant2$.
\end{theoreme}
Ce théorème a été amélioré par Gaku Liu \cite{Gakuliu}, en remplaçant \og omet $D_{4}$ et ${D_{4}}^{\star}$~\fg{} par 
\og omet~$D_{4}$ \fg{} . On obtient alors le théorème suivant dont nous donnons une nouvelle preuve. 
\begin{theoreme} \cite{Gakuliu}\label{GL et D4}
\`A isomorphisme près, Les tournois indécomposables à au moins $5$ sommets omettant $D_{4}$ sont les tournois $T_{2n+1}$, où $n\geqslant 2$.
\end{theoreme} 
\begin{proof}[\textbf{Preuve}]
Pour tout $n\geqslant2 $, le tournoi $T_{2n+1}$ omet $D_{4}$ car pour tout $x\in \lbrace 0,\ldots, 2n\rbrace$, le tournoi $T_{2n+1} [N^{+}_{T_{2n+1}}(x)]$ est transitif. Réciproquement, soit $T$  un tournoi indécomposable à au moins $5$ sommets omettant $D_{4}$. Comme $U_5$ et $W_5$ abritent $D_4$, alors $T$ omet $U_5$ et $W_5$. D'après le corollaire \ref{T,uw5 free}, $T\simeq T_{2n+1}$ pour un certain $n\geqslant 2$.  
\end{proof}
\section{Sommets fortement critiques }\label{chapitre sommets fortement critiques}
Le théorème~\ref{Gaku +1} et le corollaire \ref{cor-1,2sommets}, nous amènent à introduire la notion de sommet fortement critique comme suit. \'Etant donné un tournoi indécomposable $T$ à au moins $5$ sommets et un sommet $x\in V(T)$, on dit que $x$ est un sommet \textit{fortement critique } de $T$ lorsque pour toute partie $X\subseteq V(T)$ telle que $x\in X$, $\mid\!X\!\mid \geqslant 5$ et $T[X]$ est indécomposable, $x$ est un sommet critique de $T[X]$. Remarquons alors, que tout sommet fortement critique de $T$ est un sommet critique de $T$. La réciproque est fausse, comme le montre l'exemple suivant. Pour $n\geqslant 5$, considérons le tournoi $F_{n}$ obtenu à partir de l'ordre total $O_{n}$, en inversant les arcs $(i,i+1)$ où $0\leqslant i\leqslant n-2 $. Plus précisément, $V(F_{n})=\lbrace 0,\ldots,n-1\rbrace$ et $A(F_{n})=\lbrace (i,j)\in \lbrace 0,\ldots, n-1\rbrace^{2} : i\leqslant j+2 \rbrace \cup \lbrace (i,i-1): i\in \lbrace 1,\ldots,n-1\rbrace\rbrace $. Le tournoi $F_{n}$ vérifie les propriétés suivantes.
\begin{property}~
\begin{enumerate}
\item Pour tout $n\geqslant5$, le tournoi $F_{n}$ est indécomposable.
\item Pour tout $n\geqslant6$, tous les sommets de $F_{n}$ sont critiques sauf les sommets $0$ et $n-1$.
\item pour $n\geqslant 10$, le tournoi $F_{n}$ n'admet aucun sommet fortement critique. 
\end{enumerate}
\end{property}

\begin{proof}[\textbf{Preuve}]~~

\begin{enumerate}
\item Raisonnons par récurrence sur $n\geqslant 5$. Le tournoi $F_{5} $ est indécomposable car $F_{5}\simeq W_{5}$. Pour $n\geqslant 6$, posons $X=\lbrace 0,\ldots,n-2\rbrace$. Clairement $F_{n}[X]= F_{n-1}$. Donc d'après l'hypothèse de récurrence, $F_{n-1}$ est indécomposable. Comme $(n-1)\longrightarrow (n-2)$ et $\lbrace 0,\ldots,n-3\rbrace \longrightarrow (n-1)$, alors $(n-1)\notin \langle X\rangle$. De plus, pour tout $i\in X$, $(n-1)\notin X(i)$. Ainsi, d'après le lemme \ref{ER}, $(n-1)\in Ext(X)$.
\item Comme $F_{n}-0 \simeq F_{n}-(n-1)=F_{n-1}$, alors, d'après l'assertion $1$, les sommets $0$ et $n-1$ ne sont pas critiques.  De plus, $\lbrace 2,\ldots,n-1\rbrace$ est un intervalle de $F_{n}-1$ et pour tout $i\in\lbrace 2,\ldots,n-2\rbrace$, $\lbrace 0,\ldots i-1\rbrace$ est un intervalle non trivial de $F_{n}-i$. Ainsi, les sommets $1,\ldots,(n-2)$ sont critiques.
\item Soit $n\geqslant 10$ et $k\in \lbrace 0,\ldots,n-1\rbrace$. Au moins un des entiers $k$ ou $n-1-k$ est supérieur ou égal à $5$. Si $k\geqslant 5$, alors $F_{n}[\lbrace 0,\ldots,k\rbrace]=F_{k+1}$ est indécomposable, et $F_{n}[\lbrace 0,\ldots,k-1\rbrace]=F_{k}$ est aussi indécomposable. Il s'ensuit que le sommet $k$ n'est pas un sommet fortement critique de $F_{n}$. Si $n-1-k\geqslant 5$, alors $F_{n}[\lbrace k,\ldots,n-1\rbrace]\simeq F_{n-k}$ est indécomposable, et $F_{n}[\lbrace k+1,\ldots,n-1\rbrace]\simeq F_{n-k-1}$ est aussi indécomposable.  Il s'ensuit que le sommet $k$ n'est pas un sommet fortement critique de $F_{n}$.
 \qedhere \end{enumerate}
\end{proof}
\'Etant donné un tournoi indécomposable $T$, notons  $\mathscr{C}(T)$ \index{ $\mathscr{C}(T)$ } l'ensemble des sommets fortement critiques de $T$ et introduisons l'invariant $f(T)$ égal au nombre de sommets fortement critiques de $T$, c'est-à-dire $f(T)=\mid\! \mathscr{C}(T)\!\mid$. Dans le cas des tournois critiques, d'après le lemme \ref{T abrite T}, tous les sommets de $T$ sont fortement critiques. Dans ce paragraphe, nous montrons que le nombre de sommets fortement critiques d'un tournoi indécomposable et non critique ne dépasse pas $4$. Pour cela, nous commençons par les résultats suivants. 
\begin{lemme}\label{1.}
 Soit $T$ un tournoi indécomposable et non critique à au moins $6$ sommets et $H$ un sous-tournoi indécomposable de $T$, tel que $5\leqslant \mid\!H\!\mid< \mid\!T\!\mid$. Il existe un sous-tournoi $H'$ de $T$ tel que $H'\simeq H$ et les sommets de $V(T)\setminus V(H')$ peuvent être ordonnés $x_{1},x_{2},\ldots,x_{\mid\!T\!\mid-\mid\!H\!\mid}$, de sorte que pour tout $i\in\lbrace1,\ldots,\mid\!T\!\mid\!-\!\mid\!H\!\mid\rbrace$, $T[V(H')\cup\lbrace x_{1},\ldots,x_{i} \rbrace]$ est indécomposable et non critique.
\end{lemme}
\begin{proof}[\textbf{Preuve}]On se fixe un tournoi indécomposable $H$ à au moins $5$ sommets. Il suffit de montrer le résultat du lemme pour tout tournoi indécomposable et non critique $T$ à au moins $\mid\! H\!\mid +1$ abritant $H$. On procède par récurrence sur $\mid\!T\!\mid$. Le résultat est trivial lorsque $\mid\!T\!\mid=\mid\!H\!\mid+1$. Soit alors, un tournoi indécomposable et non critique $T$  abritant $H$ et tel que $\mid\!T\!\mid\geqslant\mid\!H\!\mid+2$. En appliquant plusieurs fois le théorème \ref{Gaku +1} à partir de $H$, on obtient une suite $K_{0}\simeq H,K_{1},\ldots,K_{\mid\!T\!\mid-\mid\!H\!\mid}=T$ de tournois indécomposables tels que pour tout $i\in\lbrace 1,\ldots,\mid\!T\!\mid\!-\!\mid\!H\!\mid\rbrace$, $K_{i}$ abrite $K_{i-1}$ et $\mid\!K_{i}\!\mid=\mid\!K_{i-1}\!\mid+1$. Le tournoi $K_{\mid\!T\!\mid-\mid\!H\!\mid-1}$ est ainsi indécomposable, non critique et abrite $H$. En appliquant l'hypothèse de récurrence au tournoi $T'=K_{\mid\!T\!\mid-\mid\!H\!\mid-1}$, il existe un sous-tournoi $H'$ de $T'$ (et donc de $T$) tel que $H'\simeq H$ et les sommets de $V(T')\setminus V(H')$ peuvent être ordonnés $x_{1},x_{2},\ldots,x_{\mid\!T\!\mid-\mid\!H\!\mid-1}$ de sorte que pour tout $i\in\lbrace1,\ldots,\mid\!T\!\mid - \mid\!H\!\mid\!-1\rbrace$, $T'[V(H')\cup\lbrace x_{1},\ldots,x_{i} \rbrace]$ est indécomposable et non critique. Comme pour tout $i\in\lbrace1,\ldots,\mid\!T\!\mid\!-\!\mid\!H\!\mid\!-1\rbrace$, $T'[V(H')\cup\lbrace x_{1},\ldots,x_{i} \rbrace]=T[V(H')\cup\lbrace x_{1},\ldots,x_{i} \rbrace]$, alors il suffit de prendre la suite $x_{1},x_{2},\ldots,x_{\mid\!T\!\mid-\mid\!H\!\mid-1},x_{\mid\!T\!\mid-\mid\!H\!\mid}$, où $x_{\mid\!T\!\mid-\mid\!H\!\mid}$ est l'élément de $V(T)\setminus V(T')$.
\end{proof}
Comme tout tournoi indécomposable à au moins $5$ sommets abrite un tournoi indécomposable à $5$ sommets, alors d'après le lemme \ref{1.}, pour tout tournoi $T$ indécomposable et non critique à au moins $5$ sommets, on a $ f(T)\leqslant 5$. Le corollaire suivant montre que la correspondance $T\longmapsto f(T)$ de la classe des tournois indécomposables et non critiques  (munie du préordre de l'abritement) est décroissante. 
\begin{corollaire}\label{f décroissante}
\'Etant donnés deux tournois $T$ et $T'$ indécomposables et non critiques à au moins $5$ sommets, si $T'$ s'abrite dans $T$, alors $f(T)\leqslant f(T')$.  
\end{corollaire}
\begin{proof}[\textbf{Preuve}]   D'après le lemme \ref{1.}, il existe un sous-tournoi $H'$ de $T$ tel que $H'\simeq T'$ et les sommets de $V(T)\setminus V(H')$ sont tous non fortement critiques. Ainsi, $\mathscr{C}(T)\subseteq V(H')$,  et donc tout sommet fortement critique de $T$ est aussi un sommet fortement critique de $H'$, c'est-à-dire $\mathscr{C}(T)\subseteq \mathscr{C}(H')$. Ainsi, $f(T)\leqslant f(H')=f(T')$.  
\end{proof}
Afin d'améliorer la borne $5$ de $f(T)$, notons les faits suivants.
\begin{Fait}\label{2.}
Tout tournoi indécomposable et non critique à au moins $6$ sommets, abrite un tournoi indécomposable à $6$ sommets.
\end{Fait}
\begin{proof}[\textbf{Preuve}]   Soit $T$ un tournoi indécomposable et non critique à au moins $6$ sommets. Si $\mid\!T\!\mid=6$, le résultat est évident et si $\mid\!T\!\mid\geqslant 7$, il suffit d'appliquer le corollaire \ref{cor-1,2sommets} plusieurs fois pour obtenir un sous-tournoi indécomposable à $6$ sommets  de $T$. 
\end{proof}
\begin{Fait}\label{3.}
Tout tournoi indécomposable $T$ à au moins $5$ sommets et admettant au plus un sommet non critique, est d'ordre impair.
\end{Fait}
\begin{proof}[\textbf{Preuve}]  
comme les tournois critiques sont d'ordres impairs, on peut supposer que $T$ admet un unique sommet non critique $x$. Le tournoi $T$ étant indécomposable à au moins $5$ sommets, alors il existe $y\neq z \in V(T)$ tels que $T[\lbrace x,y,z\rbrace]\simeq C_{3}$. Ainsi $\vert T\vert -3 $ est nécessairement pair (et donc $\vert T\vert$ est impair), car sinon, en appliquant plusieurs fois le théorème~\ref{ER +2, pour tournoi} à partir de $T[\lbrace x,y,z\rbrace]$, il existe un sommet $a\notin\{x,y,z\}$ tel que $T-a$ est indécomposable. Contradiction, puisque $x$ est l'unique sommet non critique.
\end{proof}
\begin{proposition}\label{ne dépasse pas 4}
Pour tout tournoi indécomposable et non critique $T$ à au moins $5$ sommets, on a $f(T)\leqslant 4$.
\end{proposition}
\begin{proof}[\textbf{Preuve}] 
Soit $T$ un tournoi indécomposable et non critique à au moins $5$ sommets. Les tournois indécomposables à $5$ sommets étant critiques, on a alors $\mid\!T\!\mid\geqslant 6$. D'après le fait~\ref{2.}, il existe $X\subseteq V(T)$ tel que $\mid\!X\!\mid =6$ et $T[X]$ est indécomposable. D'après le lemme~\ref{1.}, il existe $Y\subseteq V(T)$ tel que $T[Y]\simeq T[X]$ et $ \mathscr{C}(T)\subseteq Y$. Comme $T[Y]$ est indécomposable et non critique  (d'ordre pair), alors d'après le corollaire~\ref{f décroissante}, on a $f(T)\leqslant f(T[Y])$. Il suffit alors de remarquer que $f(T[Y])\leqslant 4$ d'après le fait~\ref{3.}. 
\end{proof}
Afin de prouver que la borne $4$ est optimale, nous introduisons pour tout $n\geqslant 2$, le tournoi $W_{2n+2}$ défini sur $ \lbrace 0,\ldots,2n+1\rbrace$  comme suit: $W_{2n+2}-(2n+1)=W_{2n+1}$ et $N^{+}_{W_{2n+2}}(2n+1)=\lbrace 2n-2,2n \rbrace$ (voir Figure~\ref{W_{2n+2}}). Nous vérifions que le tournoi $ W_{2n+2}$ est indécomposable et non critique. De plus, on montre que pour tout $n\geqslant 2$, on a $f(W_{2n+2})=4 $ et plus précisément,  $\mathscr{C}(W_{2n+2})=\lbrace 0,2n-2,2n-1,2n\rbrace$. Nous montrons aussi que, pour tout $n\geqslant 3$, le tournoi $W_{2n+2}-(2n-3)$ est indécomposable et non critique. Donc, d'après le corollaire~\ref{f décroissante} et la proposition~\ref{ne dépasse pas 4}, $f(W_{2n+2}-(2n-3)) =4$. Ainsi, On vient de donner deux exemples de tournois indécomposables et non critiques d'ordres arbitrairement grands (pairs et impairs) pour lesquels la borne $4$ est atteinte.  
\begin{figure}[h]
\begin{center}
\begin{tikzpicture}[line cap=round,line join=round,>=triangle 45,x=1.25cm,y=1cm]
\draw [line width=0.4pt] (10.,0.)-- (11.,0.);
\draw [->,>=latex] (10,0)-- (10.5,0);
\draw [line width=0.4pt] (11.,0.)-- (6.5,2.5);
\draw [->,>=latex] (11,0)-- (8.75,1.25);
\draw [line width=0.4pt] (9.,0.)-- (10.,0.);
\draw [->,>=latex] (9,0)-- (9.5,0);
\draw [line width=0.4pt](6.,0.)-- (7.,0.);
\draw [->,>=latex] (6,0)-- (6.5,0);
\draw [line width=0.4pt]  (3.,0.)-- (4.,0.);
\draw [->,>=latex] (3.0,0)-- (3.5,0);
\draw [line width=0.4pt] (6.5,2.5)-- (3.,0.);
\draw [->,>=latex] (6.5,2.5)-- (4.75,1.25);
\draw [line width=0.4pt] (4.,0.)-- (6.5,2.5);
\draw [->,>=latex] (4,0)-- (5.25,1.25);
\draw [line width=0.4pt] (6.5,2.5)-- (6.,0.);
\draw [->,>=latex] (6.5,2.5)-- (6.25,1.25);
\draw [line width=0.4pt] (7.,0.)-- (6.5,2.5);
\draw [->,>=latex] (7.,0)-- (6.75,1.25);\draw [line width=0.4pt] (10.,0.)-- (6.5,2.5);
\draw [->,>=latex] (6.5,2.5)-- (8.25,1.25);
\draw [line width=0.4pt] (6.5,2.5)-- (9.,0.);
\draw [->,>=latex] (9,0)-- (7.75,1.25);
\draw [line width=0.4pt] (4.,0.)-- (5.,0.);
\draw [->,>=latex] (4,0)-- (4.5,0);
\draw [line width=0.4pt] (7.,0.)-- (8.,0.);
\draw [->,>=latex] (7,0)-- (7.5,0);
\draw [line width=0.4pt] (3.,-0.6)-- (3.74,-0.6);
\draw [line width=0.4pt] (3.,-0.6)-- (3.,-0.3);
\draw [->,>=latex] (3.,-0.6)-- (3.5,-0.6);
\begin{scriptsize}
\draw [fill=black] (3.,0.) circle (3.5pt);
\draw[color=black] (2.9,-0.28) node {$0$};
\draw [fill=black] (4.,0.) circle (2pt);
\draw[color=black] (4.,-0.28) node {$1$};
\draw [fill=black] (6.,0.) circle (2pt);
\draw[color=black] (6.,-0.28) node {$2i$};
\draw [fill=black] (7.,0.) circle (2pt);
\draw[color=black] (7.2,-0.28) node {$2i+1$};
\draw [fill=black] (9.,0.) circle (2pt);
\draw[color=black] (9.0,-0.28) node {$2n-3$};
\draw [fill=black] (10.,0.) circle (3.5pt);
\draw[color=black] (10.05,-0.28) node {$2n-2$};
\draw [fill=black] (11.,0.) circle (3.5pt);
\draw[color=black] (11.22,-0.28) node {$2n-1$};
\draw [fill=black] (6.5,2.5) circle (3.5pt);
\draw[color=black] (6.5,2.78) node {$2n$};
\draw [fill=black] (5.3,0.) circle (0.5pt);
\draw [fill=black] (5.64,0.) circle (0.5pt);
\draw [fill=black] (8.24,0.) circle (0.5pt);
\draw [fill=black] (8.58,0.) circle (0.5pt);
\draw [fill=black] (10.,-1.7) circle (2pt);
\draw[color=black] (10.,-1.98) node {$2n+1$};
\draw [line width=0.912pt] (10,-1.7)-- (10,0.);
\draw [->,] (10,-1.7)-- (10,-0.7);
\draw[line width=0.8pt] (11.2,1.5) to[bend right=50] (6.5,2.5);
\draw[->,line width=0.8pt] (10,-1.7) to[bend right=45] (11.2,1.5);
\end{scriptsize}
\end{tikzpicture}  
\vspace{0.2cm} \caption{\textbf{Le Tournoi $W_{2n+2}$}}
\label{W_{2n+2}}
\end{center}
\end{figure}
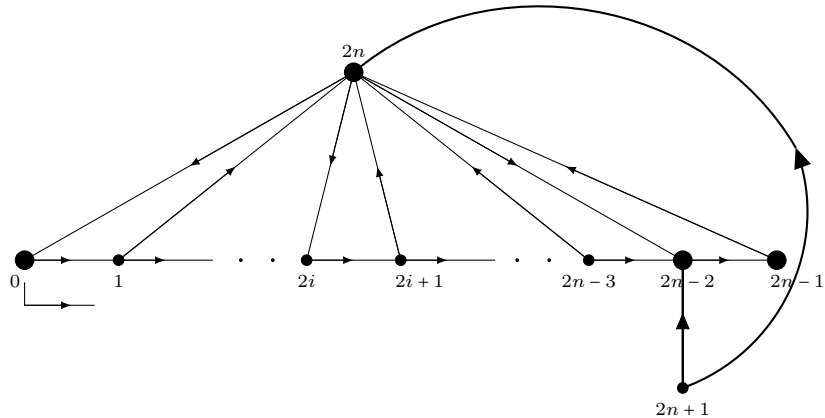

\end{document}